\documentclass[12pt]{amsart}
\usepackage[latin1]{inputenc}
\usepackage{amssymb}
\usepackage{xcolor}

\newtheorem{prop}{Proposition}[section]

\newtheorem{thm}[prop]{Theorem}
\newtheorem*{quest}{Question}
\newtheorem*{remark}{Remark}

\theoremstyle{definition}
\newtheorem*{exe}{Examples}

\renewcommand{\d}{\mathbf{d}}
\begin{document}
\noindent

\title{On the density of sumsets and product sets}

\author[Hegyv\'ari \textit{et al}]{Norbert Hegyv\'ari}
\address{Norbert Hegyv\'{a}ri, ELTE TTK,
E\"otv\"os University, Institute of Mathematics, H-1117
P\'{a}zm\'{a}ny st. 1/c, Budapest, Hungary}
\email{hegyvari@elte.hu}

\author[]{Fran\c cois Hennecart}
\address{Fran\c cois Hennecart,
Univ. Jean-Monnet, Institut Camille Jordan CNRS 5208,
23 rue Michelon, 42023 Saint-\'Etienne cedex 2, France} \email{francois.hennecart@univ-st-etienne.fr}

\author[]{P\'eter P\'al Pach $^1$}
\thanks{${}^1$ Supported by the National Research, Development and Innovation Office of Hungary
  (Grant Nr. NKFIH (OTKA) PD115978 and NKFIH (OTKA) K124171) and the J\'anos Bolyai Research
  Scholarship of the Hungarian Academy of Sciences.}
\address{P\'eter P\'al Pach,
%Department of Computer Science and Information Theory,
Budapest University of Technology and Economics, 1117 Budapest, Magyar tud\'osok k\"or\'utja 2, Hungary}

\email{ppp@cs.bme.hu}
%\thanks{Research of the authors
%are partially supported by ``Balaton Program Project" and OTKA
%grants TO 43623,49693,38396. The second author is partially
%supported by the CNRS}
\date{\today}
\maketitle

\begin{abstract}
In this paper some links between the density of a set of integers and the density of its sumset, product set and set of subset sums are presented.
\end{abstract}

\section{\bf Introduction and notations}

In the field of additive combinatorics a popular topic is to compare the densities of different sets (of, say, positive integers). The well-known theorem of Kneser  gives a description of the sets $A$ having lower density $\alpha$ such that the density of $A+A:=\{a+b\,:\, a,b\in A\}$ is less than $2\alpha$ (see for instance \cite{RO}). The analogous question with the product set $A^2:=\{ab\,:\, a,b\in A\}$ is apparently more complicated.

For any set $A\subset \mathbb{N}$ of natural numbers, we define the lower asymptotic density $\underline{\d}A$, the upper asymptotic density
$\overline{\d}A$ and the asymptotic density $\d A$ in the natural way:
$$
\underline{\mathbf{d}}A= \liminf_{n\to\infty}\frac{|A\cap[1,n]|}{n},
\quad
\overline{\mathbf{d}}A= \limsup_{n\to\infty}\frac{|A\cap[1,n]|}{n}
$$
known as the lower and upper asymptotic densities.
If the two values coincide, then we denote by $\mathbf{d}A$ the common value and call it the \textit{asymptotic density} of $A$.

Throughout the paper $\mathbb{N}$ denotes the set of positive integers and $\mathbb{N}_0:=\mathbb{N}\cup \{0\}$. We will use the notion
$A(x)=\{n\in A: n\leq x\}$ for $A\subseteq \mathbb{N}$ and $x\in \mathbb{R}$. For functions $f,g:\mathbb{N}\to \mathbb{R}_+$ we write $f=O(g)$ (or $f\ll g$), if there exists some $c>0$ such that $f(n) \leq cg(n)$ for large enough $n$.

In Section 2 we investigate the connection between the (upper-, lower-, and asymptotic) density of a set of integers and the density of its sumset. In Section 3 we give a partial answer to a question of Erd\H os by giving a necessary condition for the existence of the asymptotic density of the set of subset sums of a given set of integers. %This result strengthens a previous result of Ruzsa. \footnote{reference} 
Finally, in Section 4 we consider analogous problems for product sets.

\section{\bf Density of sumsets}

 For subsets $A,B$ of an additive monoid $G$,
the sumset $A+B$ is defined to be the set of all
sums $a+b$ with $a\in A$, $b\in B$. For $G=(\mathbb{N}_0,+)$ the following clearly hold:
\begin{align*}
\underline{\d}A&\le\overline{\d}A,\\
\underline{\d}A&\le\underline{\d}(A+A),\\
\overline{\d}A&\le\overline{\d}(A+A).
\end{align*}

\smallskip
We shall assume that our sets $A$ are normalized in the sense that $A$ contains $0$
and $\gcd(A)=1$.

First observe that there exists a set of integers $A$ not having an asymptotic density
such that its sumset $A+A$ has a density:
for instance $A=\{0\}\cup\bigcup_{n\ge0}[2^{2n},2^{2n+1}]$ has lower density $1/3$, upper density $2/3$
and its sumset $A+A$ has density $1$, since it contains every nonnegative integer.
For this kind of sets $A$, we denote
respectively
\begin{align*}
\underline{\d}A&=:\alpha_A,\\
\overline{\d}A&=:\beta_A,\\
\d(A+A)&=:\gamma_A,\\
(\alpha_A,\beta_A,\gamma_A)&=:p_A,
\end{align*}
and we have
$$
\alpha_A \le \beta_A \le \gamma_A.
$$
The first question arising from this is to decide whether or not for any
$p=(\alpha,\beta,\gamma)$ such that $0\leq \alpha\le \beta \le \gamma\leq 1$ there exists
a set $A$ of integers such that $p=p_A$. This question has no positive answer in general, though the following
weaker statement holds.

\begin{prop}\label{p1}
Let $0\le \alpha\le 1$. There exists a normalized set $A\subset\mathbb{N}$ such that
$\d A=\alpha$ and $\d(A+A)=1$.
\end{prop}

\begin{proof}
 Let $0\in B$ be a thin additive basis, that is, a basis containing 0 and satisfying $|B(x)|=o(x)$ as $x\to\infty$. Now, let $A=B\cup\{\lfloor n/\alpha\rfloor,\ n\ge1\}$. Then $A$ is a normalized set satisfying
 $A+A=\mathbb{N}_0$ and $\d A=\alpha$.
 
 (Note that $B=\{0,1,2,\dots,\lfloor 1/\alpha \rfloor\}$ is also an appropriate choice for $B$.) 
\end{proof}

\begin{remark}
We shall mention that  Faisant et al \cite{Fa} proved the following related result:  for any $0\le  \alpha \le 1$ and any positive integer $k$, there exists a sequence $A$ such that
 $\d (jA)= j\alpha /k$, $j=1,\dots,k$.
\end{remark}

After a conjecture stated by Pichorides,  the related question about the characterisation of the two-dimensional
domains $\{(\underline{\d}B,\overline{\d}B)\,:\, B\subset A\}$ has been solved
(see \cite {GG} and \cite{FH}).

\smallskip
Note that if the density   $\gamma_A$ exists, then $\alpha_A$, $\beta_A$ and $\gamma_A$  have to satisfy some strong conditions.
For instance, by Kneser's theorem, we know that if for some set $A$ we have $\gamma_A<2\alpha_A$,
then $A+A$ is, except possibly a finite number of elements, a union of
arithmetic progressions in $\mathbb{N}$ with the same difference. This implies that
$\gamma_A$ must be a rational number. From the same theorem of Kneser, we
also deduce that if $\gamma_A<3\alpha_A/2$, then $A+A$ is an arithmetic progression
from some point onward. It means that $\gamma_A$ is a unit fraction, hence $A$ contains
any sufficiently large integer, if we assume that $A$ is normalized.

Another strong connection between $\alpha_A$ and $\gamma_A$ can be deduced from
Freiman's theorem on the addition of sets (cf. \cite{Fr}). Namely, every normalized set $A$ satisfies
$$
\gamma_A\ge\frac{\alpha_A}{2}+\min\Big(\alpha_A,\frac12\Big).
$$

A related but more surprising statement is the following:

\begin{prop}\label{p3}
There is a set of positive integers for which $\d(A)$ does exist and $\d(A+A)$ does not exist.
\end{prop}

\begin{proof}
Let us take $U=\{0,2,3\}$ and $V=\{0,1,2\}$,
then observe that
$$
U+(U\cup V)=\{0,1,2,3,4,5,6\}\quad
V+(U\cup V)=\{0,1,2,3,4,5\}.
$$
Let $(N_k)_{k\ge 0}$ be a sufficiently quickly increasing sequence of integers with $N_0=0$, $N_1=1$,
and define $A$ by
$$
A=(U\cup V)\cup\bigcup_{k\ge1}\Big((U+7\mathbb{Z})\cap[7N_{2k},7N_{2k+1}]\cup (V+7\mathbb{Z})\cap[7N_{2k+1},7N_{2k+2}]\Big).
$$
Then $A$ has density $3/7$. Moreover, for any $k\ge0$
$$
[7N_{2k},7N_{2k+1}]\subset A+A,
$$
thus $\overline{\d}(A+A)=1$, if we assume $\lim_{k\to\infty}N_{k+1}/N_k=\infty$.

 \noindent
We also have
$$
(A+A)\cap [14N_{2k-1},7N_{2k}]=(\{0,1,2,3,4,5\}+7\mathbb{N})\cap
[14N_{2k-1},7N_{2k}],
$$
hence $\underline{\d}(A+A) = 6/7$ using again the assumption that  $\lim_{k\to\infty}N_{k+1}/N_k=\infty$.
\end{proof}

\medskip
%Another example due to Ruzsa is the following: let $m$ be a positive integer.

For any set $A$ having a density, let
\begin{align*}
{\d}A&=:\alpha_A,\\
\underline{\d}(A+A)&=:\underline{\gamma}_A,\\
\overline{\d}(A+A)&=:\overline{\gamma}_A,\\
(\alpha_A,\underline{\gamma}_A,\overline{\gamma}_A)&=:q_A,
\end{align*}
then we have
$$
\alpha_A\le \underline{\gamma}_A\le \overline{\gamma}_A.
$$
A question similar to the one   asked for $p_A$ can be stated as follows: given $q=(\alpha, \underline{\gamma}, \overline{\gamma})$
such that $0\le \alpha\le \underline{\gamma}\le \overline{\gamma}\leq 1$, does there exist a set $A$ such that $q=q_A$ ? 

\smallskip
We further mention an interesting question of Ruzsa: does there exist $0<\nu<1$ and a constant $c=c(\nu)>0$ such that for any set $A$ having a  density,
$$
\underline{\d}(A+A)\geq c\cdot (\overline{\d}(A+A))^{1-\nu} (\d A)^{\nu}\,?
$$
Ruzsa proved (unpublished) that in case of an affirmative answer, we  necessarily have $\nu\ge 1/2$.

\section{\bf Density of subset sums}

\medskip

Let $A=\{a_1 < a_2 < \cdots \}$ be a sequence of positive integers. Denote the set of all subset sums of $A$ by
$$
P(A):=\Big\{\sum_{i=1}^k\varepsilon_ia_i \,:\, k\ge0,\ \varepsilon_i\in \{0,1\}\ (1\le i\le k) \Big\}.
$$
Zannier conjectured and Ruzsa proved that the condition $a_n\leq 2a_{n-1}$ implies that the density $\mathbf{d}(P(A))$ exists (see \cite{r1}).
Ruzsa also asked the following questions:
\begin{itemize}
\item[i)] Is it true that for every pair of real numbers $0\leq \alpha\leq \beta\leq 1$, there exists a sequence of integers for which $\underline{\mathbf{d}}(P(A))=\alpha; \ \overline{\mathbf{d}}(P(A))=\beta$\,? This question was answered positively in \cite{r2}. 

\item[ii)] Is it true that the condition $a_n\leq a_1+a_2+\dots +a_{n-1}+c$ also implies that $\mathbf{d}(P(A))$ exists\,?

\end{itemize}

%Let us write $s_{n-1}=a_1+a_2+\dots +a_{n-1}$ and let $F(n)=|a_n-s_{n-1}|$. We shall prove the following statement.
%Let us write $s_{n-1}=a_1+a_2+\dots +a_{n-1}$. 
\noindent
We shall prove the following statement.

\begin{prop}\label{p33}
Let $(a_n)_{n=1}^\infty$ be a sequence of positive integers. Assume that for some function $\theta$ satisfying $\theta(k)\ll \frac{k}{(\log k)^2}$ we have $$|a_n-s_{n-1}|=\theta(s_{n-1})\text{ for every }n,$$
where $s_{n-1}:=a_1+a_2+\dots+a_{n-1}$.\\
\noindent
 Then $\mathbf{d}(P(A))$ exists.
\end{prop}

\begin{proof}
We first prove that there exists a real number $\delta$ such that
$$
|P(A)(s_n)|=\big(\delta+o(1)\big)s_n\quad \text{as $n \to \infty$.}
$$
Let $n\ge2$ be large enough. Then
$$
P(A)\cap[1,s_n]= \Big(P(A)\cap[1,s_{n-1}]\Big)\cup
\Big(P(A)\cap(s_{n-1},s_{n}-\theta(s_{n-1}))\Big).
$$
Since $a_n\ge s_{n-1}-\theta(s_{n-1})$, we have $P(A)\cap(s_{n-1},s_{n}]\supseteq a_n+P(A)\cap(\theta(s_{n-1}),s_{n-1}]$, thus
\begin{equation}\label{eq_low}
\Big|P(A)\cap[1,s_n]\Big|\ge 2\Big|P(A)\cap[1,s_{n-1}]\Big|-2\theta(s_{n-1})-1.
\end{equation}
On the other hand,
$$
P(A)\cap[1,s_n]\subseteq \left(P(A)\cap[1,s_{n-1}]\right)\cup \left(a_n+P(A)\cap[1,s_{n-1}]\right)\cup [s_n-\theta(s_n),s_n],
$$
since $a_{n+1}\geq s_n-\theta(s_n)$. Therefore,
\begin{equation}\label{eq_upp}
\Big|P(A)\cap[1,s_n]\Big|\le 2\Big|P(A)\cap[1,s_{n-1}]\Big|+\theta(s_{n})+1.
\end{equation}
Observe that $s_n=a_n+s_{n-1}\le 2s_{n-1}+\theta(s_{n-1})$,
hence letting
$$
\delta_n=\frac{\Big|P(A)\cap[1,s_n]\Big|}{s_n},
$$
we obtain from \eqref{eq_low} and \eqref{eq_upp} that
\begin{equation}\label{delta}
\delta_n-\delta_{n-1}=O\left(\frac{\theta(s_{n})}{s_n}\right).
\end{equation}
Now, we show that $s_n\gg 2^{n}$. Since 
\begin{equation}\label{thetabecs}
s_n=s_{n-1}+a_n\geq 2s_{n-1}-\theta(s_{n-1})=s_{n-1}\left( 2-\frac{\theta(s_{n-1})}{s_{n-1}} \right),\end{equation}
the condition $\theta(k)\ll \frac{k}{(\log k)^2}$ implies that from \eqref{thetabecs} we obtain that $s_n\gg 1.5^n$. Therefore, in fact, for large enough $n$ we have $s_n\geq s_{n-1} \left(2- \frac{c}{n^2} \right)$ with some $c>0$. Now, let $10c<K$ be a fixed integer. For $K<n$ we have 
$$s_n\geq s_K\prod\limits_{i=K+1}^{n} \left(2- \frac{c}{i^2} \right) \geq s_K \left[ 2^{n-k}-2^{n-k-1} \sum\limits_{i=K+1}^{n} \frac{c}{i^2}   \right]\gg 2^n,$$
since $\sum\limits_{i=K+1}^{n} \frac{c}{i^2}<1/10$.
Hence, $s_n\gg 2^{n}$ indeed holds.

Therefore, using the assumption on $\theta$ we obtain that $\frac{\theta(s_{n})}{s_n}\ll \frac{1}{n^2}$. So \eqref{delta} yields that
$$
\delta_n-\delta_{n-1}=O(n^{-2}).
$$
Therefore, the sequence $\delta_n$ has a limit which we denote by $\delta$.
Furthermore, observe that  
\begin{equation}\label{deltakonv}
\delta_n=\delta+O(1/n).
\end{equation}

%Let $\varepsilon>0$.
The next step is to consider an arbitrary sufficiently large positive integer $x$ and decompose it as
$$
x=a_{n_1+1}+a_{n_2+1}+\cdots+a_{n_j+1}+z,
$$
where $n_1>n_2>\cdots>n_j>k$ and $0\le z$ are defined in the following way. (Here $k$ is a fixed, sufficiently large positive integer.) The index $n_1$ is chosen in such a way that $a_{n_1+1}\le x<a_{n_1+2}$. If $x-a_{n_1+1}\geq a_{n_1}$, then $n_2=n_1-1$, otherwise $n_2$ is the largest index for which $x-a_{n_1+1}\geq a_{n_2+1}$. The indices $n_3,n_4,\dots$ are defined similarly. We stop at the point when the next index would be at most $k$ and set $z:=x-a_{n_1+1}-a_{n_2+1}-\cdots-a_{n_j+1}$. As $z\leq \theta(s_{n_1+1})+s_k$, we have 
\begin{equation}\label{zkicsi}
z=o(x).
\end{equation}

Furthermore, let
$$
b_i=a_{n_1+1}+a_{n_2+1}+\cdots+a_{n_i+1},\quad i=0,1,\dots,j.
$$
(The empty sum is $b_0:=0$, as usual.)

%$$X=[1,s_{n_1}]\cup [b_1,b_1+s_{n_2}]\cup [b_2,b_2+s_{n_3}] \cup\dots \cup [b_{j-1},b_{j-1}+s_{n_j}].$$

Let $X_0:=(0,s_{n_1}-\theta(s_{n_1}))$ and for $1\leq i\leq j-1$ let $X_i:=(b_i+\theta(s_{n_i}),b_i+s_{n_{i+1}}-\theta(s_{n_{i+1}}))$ and consider 
\begin{multline*}
X:=X_0\cup X_1 \cup \dots \cup X_{j-1}=\\
=(0,s_{n_1}-\theta(s_{n_1}))\cup (b_1+\theta(s_{n_1}),b_1+s_{n_2}-\theta(s_{n_2}))\cup %\\ \cup (b_2+\theta(s_{n_2}),b_2+s_{n_3}-\theta(s_{n_3})) \cup
\dots \cup (b_{j-1}+\theta(s_{n_{j-1}}),b_{j-1}+s_{n_j}-\theta(s_{n_j})).
\end{multline*}

%$$X=[1,s_{n_1}-\theta(s_{n_1}))\cup [b_1,b_1+s_{n_2}-\theta(s_{n_2}))\cup [b_2,b_2+s_{n_3}-\theta(s_{n_3})) \cup\dots \cup [b_{j-1},b_{j-1}+s_{n_j}-\theta(s_{n_j})).$$
Note that in this union each element  appears at most once, since according to the definition of $\theta$ the sets $X_0,X_1,\dots,X_{j-1}$ are pairwise disjoint as
$$b_i+s_{n_{i+1}}-\theta(  s_{n_{i+1}}   )\leq b_{i+1}=b_i+a_{n_{i+1}+1}$$
holds for every $0\leq i\leq j-2$.

The set of those elements of $[1,x]$ that are not covered by $X$ is:
\begin{multline*}
[1,x]\setminus X=[s_{n_1}-\theta(s_{n_1}),b_1+\theta(s_{n_1})]\cup[b_1+s_{n_2}-\theta(s_{n_2}),b_2+\theta(s_{n_2})]\cup\dots \\
\cup [b_{j-2}+s_{n_{j-1}}-\theta(s_{n_{j-1}}),b_{j-1}+\theta(s_{n_{j-1}})]\cup [b_{j-1}+s_{n_j}-\theta(s_{n_j}),x].
\end{multline*}

Therefore,
$$|[1,x]\setminus X|\leq 3\sum\limits_{i=1}^{j} \theta(s_{n_i})+z.$$

%Letting $k\to \infty$ we obtain that , 
Using $\sum\limits_{i=1}^{j} \theta(s_{n_i})\ll \sum\limits_{i=1}^{j} \frac{s_{n_i}}{n_i^2}\ll\frac{x}{k^2}$ and \eqref{zkicsi}, we obtain that $\left|[1,x]\setminus X\right|\leq (\varepsilon_k+o(1))x$, where $\varepsilon_k\to 0$ (as $k\to \infty$). (Note that $\varepsilon_k\ll 1/k^2$.)

That is, the set $X$ covers $[1,x]$ with the exception of a ``small'' portion of size $O(x/k^2)$. Therefore, by letting $k\to \infty$ the density of the uncovered part tends to 0.

%For every $0\leq i\leq j-1$ we have
%$$b_i+(P(A)\cap [0,s_{n_{i+1}}])\subseteq P(A)\cap [b_i,b_{i}+s_{n_{i+1}}],$$
%moreover, if a sum in $P(A)$ lies in the interval $[b_i,b_{i}+s_{n_{i+1}}]$, but the sum of the elements with indices larger than $n_{i+1}$ is not $b_i$, then this sum is either at most $b_i+\theta(s_{n_i})$ or at least $b_{i}+s_{n_{i+1}}-\theta(s_{n_i})$.

Let us consider $P(A)\cap X_i$. If a sum is contained in $P(A)\cap X_i$, then the sum of the elements with indices larger than $n_{i+1}$ is $b_i$. Otherwise, the sum is either at most $b_i+\theta(s_{n_i})$ or at least $b_{i}+s_{n_{i+1}}-\theta(s_{n_{i+1}})$.

Therefore, $P(A)\cap X_i=(b_i+P(\{a_1,a_2,\dots,a_{n_{i+1}}\}))\cap X_i$.

Hence,
$$\delta_{n_{i+1}}s_{n_{i+1}}-2\theta(s_{n_{i+1}})-1\leq |P(A)\cap X_i|\leq \delta_{n_{i+1}}s_{n_{i+1}}.$$

Therefore,  
\begin{multline}\label{PAalso}
|P(A)\cap[x]|\geq \sum\limits_{i=0}^{j-1} \left( \delta_{n_{i+1}}s_{n_{i+1}}-2\theta(s_{n_{i+1}})-1 \right) \geq \\
\geq \delta x -\delta z + \delta \sum\limits_{i=0}^{j-1} (s_{n_{i+1}}-a_{n_{i+1}+1})+\sum\limits_{i=0}^{j-1} (\delta_{n_{i+1}}-\delta)s_{n_{i+1}}-2\sum\limits_{i=0}^{j-1} \left(\theta(s_{n_{i+1}})+1 \right)
\end{multline} %\delta x-o(x)$$
%elements. 
and
\begin{multline}\label{PAfelso}
|P(A)\cap[x]|\leq \sum\limits_{i=0}^{j-1}  \delta_{n_{i+1}}s_{n_{i+1}} \leq \\
\leq \delta x -\delta z + \delta \sum\limits_{i=0}^{j-1} (s_{n_{i+1}}-a_{n_{i+1}+1})+\sum\limits_{i=0}^{j-1} (\delta_{n_{i+1}}-\delta)s_{n_{i+1}}
\end{multline} %\delta x-o(x)$$
Now, observe that
\begin{itemize}
    \item $|z|=o(x)$ by \eqref{zkicsi},
    \item $\sum\limits_{i=0}^{j-1} |s_{n_{i+1}}-a_{n_{i+1}+1}|=o(x)$ by using $|s_{n_{i+1}}-a_{n_{i+1}+1}|=\theta(s_{n_{i+1}})$ and $\sum\limits_{i=0}^{j-1} a_{n_{i+1}+1}\leq x$,
    \item $\sum\limits_{i=0}^{j-1} (\delta_{n_{i+1}}-\delta)s_{n_{i+1}}\ll x/k$ by using \eqref{deltakonv}. Letting $k\to \infty$ this term is also of size $o(x)$.
\end{itemize}
Hence, we obtain from \eqref{PAalso} and \eqref{PAfelso} that $|P(A)\cap[x]|=\delta x+o(x)$.

%$$
%E= [b_1,s_{n_1}]\cup[b_2,b_1+s_{n_2}]\cup\cdots\cup
%[b_{j-1},b_{j-2}+s_{n_{j-1}}].%\cup(b_j,b_{j-1}+s_{n_j}].
%$$
%(Here $[u,v]=\emptyset$, if $u>v$.)
%For the size of the set $E$ we have
%$$|E|\leq \sum\limits_{i=1}^{j-1} (s_{n_i}-a_{n_i+1})\leq \sum\limits_{i=1}^{j-1} \theta(s_{n_i})\ll \sum\limits_{i=1}^{j-1} \frac{s_{n_i}}{(\log s_{n_i})^2}=o(x),
%$$
%since $s_n\gg 2^n$.

%.
%Hence,
%$$\left|P(A)\cap [1,x]\right|=\left(\sum\limits_{i=0}^{j-1}\left |P(A)\cap [b_i,b_i+s_{n_{i+1}}]         \right| \right)+o(x).$$

%For every $0\leq i\leq j-1$ we have
%$$b_i+(P(A)\cap [0,s_{n_{i+1}}])\subseteq P(A)\cap [b_i,b_{i}+s_{n_{i+1}}],$$
%moreover, if a sum in $P(A)$ lies in the interval $[b_i,b_{i}+s_{n_{i+1}}]$, but the sum of the elements with indices larger than $n_{i+1}$ is not $b_i$, then this sum is either at most $b_i+\theta(s_{n_i})$ or at least $b_{i}+s_{n_{i+1}}-\theta(s_{n_i})$.
%The total length of the intervals $[b_i,b_i+\theta(s_{n_i})]$ and $[b_{i}+s_{n_{i+1}}-\theta(s_{n_i}),b_{i}+s_{n_{i+1}}]$ is
%$$\sum\limits_{i=1}^{j-1} 2\theta(s_{n_i})=o(x).$$
%As
%$$\left|b_i+(P(A)\cap [0,s_{n_{i+1}}])\right|=\delta s_{n_{i+1}}+O\left(\frac{s_{n_{i+1}}}{n_{i+1}}\right),$$
%we have
%$$|P(A)\cap[1,x]|=\delta %\left(\sum\limits_{i=1}^{j} s_{n_j}\right)+O\left(\sum\limits_{i=1}^{j}   \frac{s_{n_{i}}}{n_{i}}\right)+o(x)=(\delta+o(1))x,$$
%since
%$$\left| x- \sum\limits_{i=1}^{j} s_{n_j}\right|=\left| z+\sum\limits_{i=1}^{j} (s_{n_j}-x) \right|\leq |z|+\sum\limits_{i=1}^{j} \theta(s_{n_j})=o(x).$$

\end{proof}

\section{\bf Density of product sets}

For any semigroup $G$
and any subset $A\subseteq G$,  we denote by $A^2$ the product set
$$
A^2=A\cdot A=\{ab\,:\, a,b\in A\}.
$$

In this section we focus on the case $G=(\mathbb{N},\cdot)$, the semigroup (for multiplication) of all positive integers. The restricted case $G=\mathbb{N}\setminus\{1\}$ is even more interesting, since $1\in A$ implies $A\subset A^2$.

The sets of integers satisfying the small doubling
hypothesis $\mathbf{d}(A+A)=\mathbf{d}A$ are well described
through Kneser's theorem.  The similar question for the product set does not plainly lead to a strong description. We can restrict our attention to sets $A$ such that
$\gcd(A)=1$, since by setting $B:=\frac{1}{\gcd{A}}A$ we have $\mathbf{d}A=\frac{1}{\gcd(A)}\mathbf{d}B$ and $\mathbf{d}A^2=\frac{1}{(\gcd(A))^2}\mathbf{d}(B^2)$.

\begin{exe}
i) Let $A_{\mathrm{nsf}}$ be the set of all non-squarefree integers. Letting $A=\{1\}\cup A_{\mathrm{nsf}}$ we have  $A^2=A$ and
$$
\mathbf{d}A=1-\zeta(2)^{-1}.
$$
ii) However, while $\gcd(A_{\mathrm{nsf}})=1$, we have
$$
\mathbf{d}A_{\mathrm{nsf}}^2<\mathbf{d}A_{\mathrm{nsf}}=1-\zeta(2)^{-1}.
$$
iii) Furthermore, the set $A_{\mathrm{sf}}$ of all squarefree integers satisfies
$$
\mathbf{d}A_{\mathrm{sf}}=\zeta(2)^{-1} \text{ and }
\mathbf{d}A_{\mathrm{sf}}^2=\zeta(3)^{-1},
$$
since $A_{\mathrm{sf}}^2$ consists of all cubefree integers.\\[0.5em]
iv) Given a positive integer $k$, the set $A_k=\big\{n\in\mathbb{N}\,:\,
\gcd(n,k)=1\big\}$ satisfies
$$
A_k^2=A_k\quad\text{and}\quad \mathbf{d}A_k=\frac{\phi(k)}k,
$$
where $\phi$ is Euler's totient function.
\end{exe}

We have the following result:

\begin{prop}
For any positive $\alpha<1$ there exists a set $A\subset \mathbb{N}$ such that
$\mathbf{d}A>\alpha$ and  $\mathbf{d}A^2<\alpha$.
\end{prop}

\begin{proof}
Assume first that $\alpha<1/2$.\\
For $k\ge 1$ let $A_k=k\mathbb{N}=\{kn,\ n\ge1\}$, then $A_k^2=k^2\mathbb{N}$. Therefore, $\mathbf{d} A_k=1/k$ and $\mathbf{d} (A_k^2)=1/k^2$.
If $1/(k+1) \le \alpha < 1/k$, then $A_k$ satisfies the requested condition. Since $\bigcup\limits_{k\geq 2}\left[ \frac{1}{k+1},\frac{1}{k} \right)=(0,1/2)$, an appropriate $k$ can be chosen for every $\alpha\in (0,1/2).$ \\[0.5em]
Assume now that $1>\alpha\ge1/2$.\\
Let $p_1<p_2<\cdots$ be the increasing sequence of prime numbers and
$$
B_r:=\bigcup_{i=1}^rp_i\mathbb{N}.
$$
The complement of the set $B_r$ contains exactly those positive integers that are not divisible by any of $p_1,p_2,\dots,p_r$, thus we have
$$\mathbf{d}(B_r)=1-\prod_{i=1}^r\left(1-\frac1{p_i}\right)=:\gamma_r.$$
Similarly, the complement of the set $B_r^2$ contains exactly those positive integers that are not divisible by any of $p_1,\dots,p_r$ or can be obtained by multiplying such a number by one of $p_1,\dots,p_r$. Hence, we obtain that
$$\mathbf{d}(B_r^2)=1-\left(1+\sum\limits_{i=1}^r \frac{1}{p_i}\right)\prod_{i=1}^r\left(1-\frac1{p_i}\right)=:\beta_r.$$
Note that 
\begin{equation}\label{intersect}
\beta_{r+1}=1-\left(1+\sum\limits_{i=1}^{r+1} \frac{1}{p_i}\right)\left(1-\frac{1}{p_{r+1}}\right)\prod_{i=1}^r\left(1-\frac1{p_i}\right)<1-\frac{3}{2}\cdot\frac{2}{3}\cdot\prod_{i=1}^r\left(1-\frac1{p_i}\right)=\gamma_r.
\end{equation}
As $(\beta_1,\gamma_1)=(1/4,1/2)$, moreover $(\beta_r)_{r=1}^\infty$ and $(\gamma_r)_{r=1}^\infty$ are increasing sequences satisfying \eqref{intersect} and $\lim \gamma_r =1$, we obtain that $[1/2,1)$ is covered by $\bigcup\limits_{r=1}^\infty (\beta_r,\gamma_r)$. That is, for every $\alpha\in[1/2,1)$ we have $\alpha\in (\beta_r,\gamma_r)$ for some $r$, and then $A=B_r$ is an appropriate choice.

\end{proof}

We pose two questions about the densities of $A$ and $A^2$.

\begin{quest}
 If $1\in A$ and $\mathbf{d}A=1$, then $\mathbf{d}(A^2)=1$, too.
Given two integers $k,\ell$, the set
 $$
 \{n\in\mathbb{N}\,:\, \gcd(n,k)=1\}\cup k\ell \mathbb{N}
 $$
 is multiplicatively stable.
What are the sets $A$ of positive integers such that $A^2=A$ or less
restrictively
$$
1\in A\text{ and }1>\mathbf{d}A^2=\mathbf{d}A>0\,?
$$
\end{quest}

\begin{quest}
It is clear that  $\mathbf{d}A>0$ implies  $\mathbf{d}A^2>0$, since
$A^2\supset (\min A)A$.

 For any $\alpha\in(0,1)$ we denote
 $$
 f(\alpha):=\inf_{A\subset\mathbb{N}; \ \mathbf{d}A=\alpha}\mathbf{d}A^2.
 $$
 Is it true that $f(\alpha)=0$ for any $\alpha$ or at least for $\alpha<\alpha_0$\,?
\end{quest}

The next result shows that the product set of a set having density 1 and satisfying a technical condition must also have density 1.
%The next result gives a tentative extension of the above assertion\footnote{which assertion?}. %Let us recall  that if $\underline{\mathbf{d}}A>0$, then $\sum_{a\in A}a^{-1}=\infty$.

\begin{prop}\label{pp1}
Let $1\notin A$ be a set of positive integers with asymptotic density $\mathbf{d}A=1$.
Furthermore, assume that $A$ contains %a sufficiently dense 
an infinite subset of mutually coprime integers $a_1<a_2<\cdots$ such that%, namely
$$
\sum_{i\ge1}\frac{1}{a_i}=\infty.
$$
Then the product set $A^2$ also has density $\mathbf{d}(A^2)=1$.
\end{prop}

\begin{proof}
Let $\varepsilon>0$ be arbitrary and choose a large enough $k$ such that
\begin{equation}\label{ep}
\sum_{i=1}^k \frac1a_i >-\log\varepsilon.
\end{equation}
Let $x$ be a large integer.
For any $i=1,\dots,k$, the set $A^2(x)$ contains all the products $a_ia$ with $a\le x/a_i$. We shall use a sieve argument. Let $A'$ be a finite subset of $A$ and  $X=[1,x]\cap\mathbb{N}$ for some $x>\max(A')$. For any $a\in A'$, let
$$
X_a=\Big\{n\le x\,:\, a\nmid n\text{ or }\frac na\not \in A\Big\}.
$$
Observe that
$$
X\setminus X_a=(aA)(x).
$$
Then
$$
(A'A)(x)=\bigcup_{a\in A'}\left(X\setminus X_a\right).
$$
By the inclusion-exclusion principle we obtain
$$
|(A'A)(x)|=\sum_{k=1}^{|A'|}(-1)^{j-1}\sum_{\substack{B\subseteq A'\\|B|=j}}\Big|\bigcap_{b\in B}\left(X\setminus X_b\right)\Big|,
$$
whence
\begin{equation}\label{szita}
\Big|\bigcap_{a\in A'}X_a\Big|= \sum_{j=0}^{|A'|}(-1)^{j}\sum_{\substack{B\subseteq A'\\|B|=j}}\Big|\bigcap_{b\in B}\left(X\setminus X_b\right)\Big|,
\end{equation}
where the empty intersection $\bigcap_{b\in \emptyset}\left(X\setminus X_b\right)$ denotes the full set $X$.

For any finite set of integers $B$ we denote by $\mathrm{lcm}(B)$ the least common multiple of the elements of $B$.
Now, we consider
$$
\bigcap_{b\in B}\left(X\setminus X_b\right)=\Big\{n\le x\,:\, \mathrm{lcm}(B)\mid n\text{ and }  \frac{n}{b}\in A \ (\forall b\in B) \Big\}.
$$
By the assumption $\mathbf{d}A=1$ we immediately get
$$
\Big|\bigcap_{b\in B}\left(X\setminus X_b\right)\Big|=\frac{x}{\mathrm{lcm}(B)}(1+o(1)).
$$
Plugging this into \eqref{szita}:
$$
\Big|\bigcap_{a\in A'(x)}X_a\Big|= x\sum_{k=0}^{|A'|}(-1)^{j}\sum_{\substack{B\subseteq A'\\|B|=j}}\frac1{\mathrm{lcm}(B)}+o(x).
$$

%We obtain for $A'=\{a_1,\dots,a_k\}$, where the $a_i$'s are mutually coprime
Since the elements of $A'=\{a_1,a_2,\dots,a_k\}$ are mutually coprime,
$$
x- |A'A(x)| = x\sum_{j=0}^k(-1)^j\sum_{1\le a_{i_1}<\cdots<a_{i_j}\le k}
\frac{1}{a_{i_1}a_{i_2}\dots a_{i_j}}+o(x)=x\prod_{i=1}^k\Big(1-\frac1{a_i}\Big)+o(x).
$$
(Note that for $j=0$ the empty product is defined to be 1, as usual.)
Since $1-u\le \exp(-u)$ we get
$$
x- |A'A(x)|  \le x\exp\Big(-\sum_{i=1}^k \frac1{a_i}\Big)+o(x)<\varepsilon x+o(x)
$$
by our assumption \eqref{ep}. Thus finally
$$
|A^2(x)|\ge |A'A(x)|>x(1-\varepsilon-o(1)).
$$
This ends the proof.
\end{proof}

\begin{remark}
Specially, the preceding result applies when $A$ contains a %sufficiently dense 
sequence of prime numbers $p_1<p_2<\cdots$ %in order to have a divergent series
such that
$\sum_{i\ge1}1/p_i=\infty$.
For this it is enough to assume that
$$
\liminf_{i\to\infty}\frac{i\log i}{p_i}>0.
$$
\end{remark}

However, we do not know how to avoid the assumption on the mutually coprime integers having infinite reciprocal sum. We thus pose the following question:

\begin{quest}
Is it true that $\mathbf{d}A=1$ implies $\mathbf{d}(A^2)=1$?
\end{quest}

\medskip
%If we only assume that $\gcd(a_k,a_1)=1$ then $a'_1=a_1$. Then $\alpha'_{k-1}\ge 1/a_1$.

\subsection*{An example for a set $A$ such that $\mathbf{d}(A)=0$ and $\mathbf{d}(A^2)=1$.}

According to the fact that the multiplicative properties of the elements play an important role, one can build a set whose elements are characterized by their number of prime factors. Let
$$
A=\big\{n\in\mathbb{N}\,:\,\Omega(n)\le 0.75 \log\log n+1\big\},
$$
where $\Omega(n)$ denotes the number of prime factors (with multiplicity) of $n$. %and $C$ is a sufficiently large constant.
An appropriate generalisation of the Hardy-Ramanujan theorem (cf. \cite{HR} and \cite{GT}) shows that the normal order of $\Omega(n)$ is $\log\log n$ and the Erd\H os-Kac theorem asserts that
$$
\mathbf{d}\left\{n\in\mathbb{N}\,:\, \alpha<\frac{\Omega(n)-\log\log n}{\sqrt{\log\log n}}<\beta\right\}=\frac1{\sqrt{2\pi}}\int_{\alpha}^{\beta}e^{-t^2/2}dt,
$$
which implies $\mathbf{d}A=0$. Now we prove that $\mathbf{d}A^2=1$. The principal feature in the definition of $A$ is that $A^2$  must contain almost all integers $n$ such that $\omega(n)\le 1.2 \log\log n$.

For $n\in\mathbb{N}$ let
$$
P_+(n):=\max\big\{p: p\text{ is a prime divisor of } n\big\}.
$$
Let us consider first the density of the integers $n$ such that
\begin{equation}\label{p+}
P_+(n) > n\exp(-(\log n)^{4/5}).
\end{equation}
Let $x$ be a large number and write
$$
\left|
\Big\{n\le x\,:\, P_+(n)\le n\exp(-(\log n)^{4/5})
\Big\}\right|=\left|
\Big\{n\le x\,:\, P_+(n)\le x\exp(-(\log x)^{4/5})
\Big\}\right|+o(x).
$$
By a theorem of Hildebrand (cf. \cite{Hi}) on the estimation  of $\Psi(x,z)$, the number of $z$-friable integers  up to
$x$, we conclude that the above cardinality is $x+o(x)$.
Hence, we may avoid the integers $n$ satisfying \eqref{p+}. By the same estimation we may also avoid those integers $n$ for which $P_+(n) < \exp((\log n)^{4/5})$.

Let $n$ be an integer such that $\Omega(n)\leq 1.2\log\log n$ and
$$
\exp((\log n)^{4/5})\le P_+(n)\le n\exp(-(\log n)^{4/5}).
$$
Our goal is to find a decomposition $n=n_1n_2$ with $\Omega(n_i)\le 0.75\log\log n_i+1$, $i=1,2$.

Let
$$
n=p_1p_2\dots p_{t-1}P_+(n),
$$
where $t=\Omega(n)$. We also assume that $p_1\le p_2\le \cdots\le p_{t-1}\le P_+(n)$. Let $m=\frac{n}{P_+(n)}$. Then
$$
\exp((\log n)^{4/5})\le m\le n\exp(-(\log n)^{4/5}).
$$
Let
$$
n_1=p_1p_2\dots p_{u-1}P_+(n)\text{ and }
n_2=p_u\dots p_{t-1},
$$
where $u=\lfloor(t-1)/2\rfloor$. Then
$n_2\ge \sqrt{m}$, which yields
$$
\log\log n_2\ge \log\log m -\log 2\ge 0.8\log\log n-\log 2.
$$
On the other hand,
$$
\Omega(n_2)=t-u\le \frac t2+1 \le 0.6 \log\log n+1\le 0.75 \log\log n_2+\frac{3\log 2}{4}.
$$
Now $n_1\ge  P_+(n)\ge \exp((\log n)^{4/5})$, hence
$$
\log\log n_1\ge 0.8\log\log n
$$
and
$$
\Omega(n_1)\le \frac {t-1}{2}\le 0.6 \log\log n\le 0.75 \log\log n_1$$
Therefore, the following statement is obtained:

\begin{prop}\label{pp3}
The set
$$
A=\big\{n\in\mathbb{N}\,:\,\Omega(n)\le 0.75 \log\log n+1\big\}
$$
has  density $0$ and its product set $A^2$ has density $1$.
\end{prop}

By a different approach we may extend the above result as follows.

\begin{thm}
For every $0\leq \alpha\leq \beta\leq 1$ there exists a set $A\subseteq \mathbb{N}$ such that $\mathbf{d}A=0$, $\underline{\d}(A\cdot A)=\alpha$ and $\overline{\d}(A\cdot A)=\beta$.
\end{thm}

\begin{proof}
We start with defining a set $Q$ such that $\d(Q)=0$ and $ \d(Q\cdot Q)=\beta$.
Let us choose a subset $P_0$ of the primes such that $\prod\limits_{p\in P_0}(1-1/p)=\beta$. Such a subset can be chosen, since $\sum 1/p=\infty$.
Now, let $p_k$ denote the $k$-th prime and let
$$
P_1=\{p_i: i\text{ is odd}\}\setminus P_0,
$$
$$
P_2=\{p_i: i\text{ is even}\}\setminus P_0.
$$
Furthermore, let
 $$
 Q_1=\{n: \text{all prime divisors of $n$ belong to $P_1$}\}
 $$
 and
 $$
 Q_2=\{n: \text{all prime divisors of $n$ belong to $P_2$}\}.
 $$
 Let $Q=Q_1\cup Q_2$. Clearly, $Q\cdot Q=Q_1\cdot Q_2$ contains exactly those numbers that do not have any prime factor in $P_0$, so $\d(Q\cdot Q)=\beta$. For $i\in\{1,2\}$ and $x\in\mathbb{R}$ the probability that an integer does not have any prime factor being less than $x$ from $P_i$ is $\prod\limits_{p< x, p\in P_i}(1-1/p)\leq \frac{1}{\beta}\prod\limits_{p< x, p\in P_i\cup P_0}(1-1/p)\leq \frac{1}{\beta}\exp\left\{-\sum\limits_{\substack{j:\ p_j< x, \\ j\equiv i\pmod{2}}}\frac{1}{p_j}\right\}=O\left( \frac{1}{\beta\sqrt{\log x}}\right)$. Therefore, $\d(Q_1)=\d(Q_2)=0$, and consequently $\d(Q)=0$ also holds.
  If $\alpha=\beta$, then $A=Q$ satisfies the conditions. From now on let us assume that $\alpha<\beta$.

Our aim is to define a subset $A\subseteq Q$ in such a way that $\underline{\d}(A\cdot A)=\alpha$ and $ \overline{\d}(A\cdot A)=\beta$. As $A\subseteq Q$ we will have  $\d(A)=0$ and $\overline{\d}(A\cdot A)\leq \beta$. The set $A$ is defined recursively. We  will define an increasing sequence of integers $(n_j)_{j=1}^\infty$ and sets $A_j$ ($j\in\mathbb{N}$)  satisfying the following conditions (and further conditions to be specified later):
\begin{itemize}
\item[(i)] $A_j\subseteq A_{j-1}$,
\item[(ii)] $A_j\cap [1,n_{j-1}]=A_{j-1}\cap [1,n_{j-1}]$,
\item [(iii)]$A_j\cap [n_j+1,\infty]=Q\cap [n_j+1,\infty]$.
\end{itemize}
That is, $A_j$ is obtained from $A_{j-1}$ by dropping out some elements of $A_{j-1}$ in the range $[n_{j-1}+1,n_j]$.
Finally, we  set $A=\bigcap\limits_{j=1}^\infty A_j$.

Let $n_1=1$ and  $A_1=Q$. %If $A_i$ is already defined, then either $A_{i+1}=A_i$ or $A_{i+1}=A_i\setminus \{i+1\}$.
% Note that $A(i)=A_i$ for every $i$.
%That is, for every element of $Q$ we will decide whether we keep it or drop it out.
We  define the sets $A_j$ in such a way that the following condition holds for every $j$ with some $n_0$ depending only on $Q$:
$$
|(A_j\cdot A_j)(n)|\geq \alpha n \text{ for every $n\geq n_0$.}\leqno(*)
$$

Since $d(Q\cdot Q)=\beta>\alpha$, a threshold $n_0$ can be chosen in such a way that $(*)$ holds for $A_1=Q$ with this choice of $n_0$.
%An increasing sequence of positive integers $1=n_1<n_2<\dots$ is also going to be defined recursively.
Now, assume that $n_j$ and $A_j$ are already defined for some $j$.
We continue in the following way depending on the parity of $j$:
%We define $n_{j+1}$ and $A_{j+1}$ in the following way, depending on the parity of $j$:\\
\begin{itemize}
\item[Case I:] $j$ is odd.

%{Case 1: $j$ is odd.}\\
Let $n_j<s$ be the smallest integer such that
$$|(A_j\setminus [n_j+1,s])\cdot (A_j\setminus [n_j+1,s])(n)|<\alpha n$$
for some $n\geq n_0$. We claim that such an $s$ exists, indeed it is at most $\lfloor n_j^2/\alpha \rfloor +1$.
For $s'=\lfloor n_j^2/\alpha \rfloor +1$ we have
$$|(A_j\setminus [n_j+1,s'])\cdot (A_j\setminus [n_j+1,s'])(s')|\leq n_j^2<\alpha s'.$$
Hence, $s$ is well-defined (and $s\leq s'$). Let $n_{j+1}:=s-1$ and $A_{j+1}:=A_j\setminus [n_j+1,s-1]$. (Specially, it can happen that $n_{j+1}=n_j$ and $A_{j+1}=A_j$.) Note that $A_{j+1}$ satisfies $(*)$.\\

\item[Case II:] $j$ is even.

%{Case 2: $j$ is even.}\\
%Assume that $A_i$ is defined for $n_j\leq i\leq l-1$. %Let $A_l:=A_{l-1}$, if $|(A_l\setminus \{l\})\cdot (A_l\setminus \{l\})|%<(\beta-1/j)l$.
Now, let $n_j< s$ be the smallest index for which $|(A_{j}\cdot A_{j})(s)|>(\beta-1/j)s$.

We have $\d(Q\cdot Q)=\beta$ and $A_{j}$ is obtained from $Q$ by deleting finitely many elements of it: $A_{j}=Q\setminus R$, where $R\subseteq [n_j]$. As $\d(Q)=0$, we have that
$$|((Q\cdot Q)\setminus (Q\setminus R)\cdot (Q\setminus R))(n)|\leq |R|^2+\sum\limits_{r\in  R} |Q(n/r)|=o(n),$$
 therefore, $\d(A_{j}\cdot A_{j})=\beta$. So for some $n>n_j$ we have that $(A_{j}\cdot A_{j})(n)>(\beta-1/j)n$, that is, $s$ is well-defined. Let $n_{j+1}:=s$ and $A_{j+1}=A_j$. Clearly, $A_{j+1}$ satisfies $(*)$.

\end{itemize}

This way an increasing sequence $(n_j)_{j=1}^\infty$ and sets $A_j (j\in \mathbb N)$ are defined, these satisfy conditions (i)-(iii). Finally, let us set $A:=\bigcap\limits_{j=1}^\infty A_j$. Note that $A(n_j)=A_j(n_j)$.

We have already seen that $A\subseteq Q$ implies that  $\d(A)=0$ and $\overline{\d}(A\cdot A)\leq \beta$. At  first we show that $\underline{\d}(A\cdot A)\geq \alpha$. Let $n\geq n_0$ be arbitrary. If $j$ is large enough, then $n_j>n$. As $A_j$ satisfies $(*)$ and $(A\cdot A)(n)=(A_j\cdot A_j)(n)$ we obtain that
$$|(A\cdot A)(n)|=|(A_j\cdot A_j)(n)|\geq \alpha n.$$
This holds for every $n\geq n_0$, therefore, $\underline{\d}(A\cdot A)\geq \alpha$.

As a next step, we show that $\underline{\d}(A\cdot A)=\alpha$. Let $j$ be odd. According to the definition of $n_{j+1}$ and $A_{j+1}$ there exists some $n\geq n_0$ such that
$$|((A_j\setminus \{n_{j+1}+1\})\cdot (A_j\setminus \{n_{j+1}+1\}))(n)|<\alpha n. $$
For brevity, let $s:=n_{j+1}+1$. As $A\subseteq A_j$ we get that $|(A\setminus \{s\})\cdot (A\setminus \{s\})(n)|<\alpha n$. Also, $$|(A\cdot A)\setminus ((A\setminus\{s\})\cdot (A\setminus\{s\})(n))|\leq 1+|A(n/s)| \leq 1+|Q(n/s)|,$$
since $A\subseteq Q$.  Thus $|(A\cdot A)(n)|\leq \alpha n+1+|Q(n/s)|\leq n(\alpha+1/n+1/s)$.
Clearly $s=n_{j+1}+1\leq n$, and as $j\to \infty$ we have $n_{j+1}\to \infty$, therefore $\underline{\d}(A\cdot A)=\alpha$.

Finally, we prove that $\overline{\d}(A\cdot A)=\beta$. Let $j$ be even. According to the definition of $A_{j+1}$ and $n_{j+1}$, we have $|(A_{j+1}\cdot A_{j+1})(n_{j+1})|>(\beta-1/j)n_{j+1}$. However, $(A\cdot A)(n_{j+1})=(A_{j+1}\cdot A_{j+1})(n_{j+1})$, therefore $\overline{\d}(A\cdot A)\geq \lim (\beta-1/j)=\beta$, thus $\overline{\d}(A\cdot A)=\beta$ as it was claimed.

\end{proof}

\end{document}